\documentclass{amsart}

\usepackage{amsthm}
\usepackage{mathrsfs}
\usepackage{proof}
\usepackage{url}
\usepackage[matrix,arrow,graph,curve]{xy}

\newtheorem{theorem}{Theorem}[section]
\newtheorem{proposition}[theorem]{Proposition}
\newtheorem{lemma}[theorem]{Lemma}

\theoremstyle{definition}

\theoremstyle{remark}
\newtheorem{remark}[theorem]{Remark}

\newcommand{\A}{\ensuremath{\mathscr{A}}}
\newcommand{\B}{\ensuremath{\mathscr{B}}}
\newcommand{\X}{\ensuremath{\mathscr{X}}}
\newcommand{\Z}{\ensuremath{\mathcal{Z}}}
\newcommand{\V}{\ensuremath{\mathscr{V}}}

\newcommand{\C}{\ensuremath{\mathscr{C}}}
\renewcommand{\S}{\ensuremath{\mathscr{S}}}
\newcommand{\D}{\ensuremath{\mathscr{D}}}
\newcommand{\E}{\ensuremath{\mathscr{E}}}
\newcommand{\W}{\ensuremath{\mathscr{W}}}
\newcommand{\M}{\ensuremath{\mathscr{Ms}}}

\newcommand{\ox}{\ensuremath{\otimes}}
\newcommand{\ra}{\ensuremath{\xymatrix@1@C=15pt{\ar[r]&}}}
\newcommand{\Ve}{\ensuremath{\overset{\vee}}}
\newcommand{\PA}{\ensuremath{\mathcal{P} A}}
\newcommand{\QA}{\ensuremath{\mathcal{Q} A}}
\renewcommand{\r}{\ensuremath{\mathscr{R}}}
\newcommand{\T}{\ensuremath{\mathscr{T}}}

\newcommand{\Mkyf}{\ensuremath{\mathbf{Mky}_\textit{fin}}}
\newcommand{\Comod}{\ensuremath{\mathbf{Comod}}}

\newcommand{\GSet}{\ensuremath{G\text{-}\mathbf{Set}}}
\newcommand{\Spn}{\ensuremath{\mathbf{Spn}}}
\newcommand{\CO}{\ensuremath{\mathbf{\small{COCT}}}}
\newcommand{\ModA}{\ensuremath{\mathbf{Mod}(A)}}
\newcommand{\ModB}{\ensuremath{\mathbf{Mod}(B)}}

\newcommand{\Mod}{\ensuremath{\mathbf{Mod}}}
\newcommand{\kMod}{\ensuremath{\mathbf{Mod}_k}}
\newcommand{\CMon}{\ensuremath{\mathbf{CMon}}}
\newcommand{\VMod}{\ensuremath{\mathscr{V}\text{-}\mathbf{Mod}}}
\newcommand{\CAT}{\ensuremath{\mathbf{CAT}}}

\newcommand{\Mky}{\ensuremath{\mathbf{Mky}}}
\newcommand{\Vect}{\ensuremath{\mathbf{Vect}}}
\newcommand{\Rep}{\ensuremath{\mathbf{Rep}}}
\newcommand{\Grn}{\ensuremath{\mathbf{Grn}}}
\newcommand{\Hom}{\ensuremath{\mathrm{Hom}}}

\renewcommand{\o}{\ensuremath{\circ}}

\newcommand{\Ra}{\ensuremath{\xymatrix@1{\ar@{=>}[r]&}}}
\newcommand{\lr}{\ensuremath{\xymatrix@1{\ar@{<->}[r] |-{\object@{|}} & }}}

\begin{document}
\title{Mackey functors on compact closed categories}
\thanks{The authors are grateful for the support of the Australian Research Council Discovery Grant DP0450767, and the first author for the support of an Australian International Postgraduate Research Scholarship, and an International Macquarie University Research Scholarship.}
\author{Elango Panchadcharam}
\author{Ross Street}
\email{\{elango, street\}@maths.mq.edu.au}
\address{Centre of Australian Category Theory, Macquarie University, New South Wales, 2109, Australia}
\dedicatory{Dedicated to the memory of Saunders Mac Lane}
\begin{abstract}
We develop and extend the theory of Mackey functors as an application of enriched category theory. We define
Mackey functors on a lextensive category $\E$ and investigate the properties of the category of Mackey
functors on $\E$. We show that it is a monoidal category and the monoids are Green functors. Mackey functors are
seen as providing a setting in which mere numerical equations occurring in the theory of groups can be
given a structural foundation. We obtain an explicit description of the objects of the Cauchy
completion of a monoidal functor and apply this to examine Morita equivalence of Green functors.
\end{abstract}
\maketitle


\section{Introduction}
Groups are used to mathematically understand symmetry in nature and in mathematics itself. Classically,
groups were studied either directly or via their representations. In the last 40 years, arising from the latter,
groups have been studied using Mackey functors.

Let $k$ be a field. Let $\Rep(G) = \Rep_k(G)$ be the category of $k$-linear representations of the
finite group $G$. We will study the structure of a monoidal category $\Mky(G)$ where the objects are called
Mackey functors. This provides an extension of ordinary representation theory. For example,
$\Rep(G)$ can be regarded as a full reflective sub-category of $\Mky(G)$; the reflection is strong
monoidal (= tensor preserving).
Representations of $G$ are equally representations of the group algebra $kG$; Mackey
functors can be regarded as representations of the "Mackey algebra" constructed from $G$.
While $\Rep(G)$ is compact closed (= autonomous monoidal), we are only able to show that
$\Mky(G)$ is star-autonomous in the sense of \cite{Ba}.

Mackey functors and Green functors (which are monoids in $\Mky(G)$) have been studied fairly
extensively. They provide a setting in which mere numerical equations occurring in group theory
can be given a structural foundation. One application has been to provide relations between $\lambda$- and $\mu$-invariants in Iwasawa theory and between Mordell-Weil groups, Shafarevich-Tate groups,
Selmer groups and zeta functions of elliptic curves (see \cite{BB}).

Our purpose is to give the theory of Mackey functors a categorical simplification and generalization. We
speak of Mackey functors on a compact (= rigid = autonomous) closed category $\T$. However, when $\T$ is
the category  $\Spn(\E)$ of spans in a lextensive category $\E$, we speak of Mackey functors on $\E$.
Further, when $\E$ is the category (topos) of finite $G$-sets, we speak of Mackey functors on $G$.

Sections \ref{Se2}-\ref{Se4} set the stage for Lindner's result \cite{Li1} that Mackey functors, a concept
going back at least as far as \cite{Gr}, \cite{Dr1} and \cite{Di} in group representation theory,
can be regarded as functors out of the category of spans in a suitable category $\E$. The important
property of the category of spans is that it is compact closed. So, in Section \ref{Se5}, we look
at the category $\Mky$ of additive functors from a general compact closed category $\T$
(with direct sums) to the category of $k$-modules. The convolution monoidal structure
on $\Mky$ is described; this general construction (due to Day \cite{Da}) agrees with the usual
tensor product of Mackey functors appearing, for example, in \cite{Bo1}. In fact, again for general
reasons, $\Mky$ is a closed category; the internal hom is described in Section \ref{Se6}. Various
convolution structures have been studied by Lewis \cite{Le} in the context of Mackey functors for compact
Lie groups mainly to provide counter examples to familiar behaviour.

Green functors are introduced in Section \ref{Se7} as the monoids in $\Mky$. An easy construction,
due to Dress \cite{Dr1}, which creates new Mackey functors from a given one, is described in
Section \ref{Se8}. We use the (lax) centre construction for monoidal categories to explain
results of \cite{Bo2} and \cite{Bo3} about when the Dress construction yields a Green functor.

In Section \ref{Se9} we apply the work of \cite{Da4} to show that finite-dimensional Mackey
functors form a $*$-autonomous \cite{Ba} full sub-monoidal category $\Mkyf$ of $\Mky$.

Section \ref{Se11} is rather speculative about what the correct notion of Mackey functor should be
for quantum groups.

Our approach to Morita theory for Green functors involves even more serious use of enriched
category theory: especially the theory of (two-sided) modules. So Section \ref{Se12} reviews this theory
of modules and Section \ref{Se13} adapts it to our context. Two Green functors are Morita equivalent
when their $\Mky$-enriched categories of modules are equivalent, and this happens, by the
general theory, when the $\Mky$-enriched categories of Cauchy modules are equivalent.
Section \ref{Se14} provides a characterization of Cauchy modules.

\section{The compact closed category $\Spn(\E)$ } \label{Se2}
Let $\E$  be a finitely complete category. Then the category $\Spn(\E)$
can be defined as follows. The objects are the objects of the category $\E$ and
morphisms  $U \ra V$ are the isomorphism classes of \emph{spans} from $U$ to $V$
in the bicategory of spans in $\E$ in the sense of \cite{Be}. (Some properties of this bicategory can
be found in \cite{CKS}.) A \emph{span} from $U$ to $V$, in the sense
of \cite{Be}, is a diagram of two morphisms with a common domain $S$, as in
\[
(s_1, S, s_2): \quad
\vcenter{\xygraph{
S="s"
(:[d(1)r(0.9)] {V~.}^-{s_2},
:[d(1)l(0.9)] {U}_-{s_1}
)}}
\]
An isomorphism of two spans $(s_1, S, s_2): U \ra V$ and  $(s'_1, S', s'_2): U \ra V$
is an invertible arrow $h: S \ra S'$ such that $s_1=s'_1 \o h$ and $s_2=s'_2 \o h$. \[\xygraph{
S="s"
(:[d(1)r(1.3)] {V}="b"^-{s_2},
:[d(1)l(1.3)] {U}="a"_-{s_1},
:[d(2.1)] {S'}="c"^-{h}_{\cong}
"c" : "a"^-{s'_1}
"c" : "b"_-{s'_2}
)}
\]
The composite of two spans $(s_1, S, s_2): U \ra V$ and  $(t_1, T, t_2): V \ra W$ is defined to be
$(s_1\o {\text{proj}}_1, T\o S, t_2 \o {\text{proj}}_2): U \ra W$ using the pull-back diagram
as in
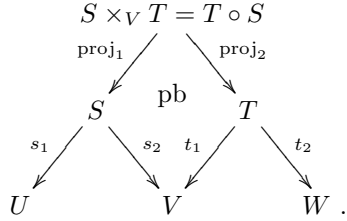
\[\xygraph{
{S \times_V T = T \o S}="a"
(:[d(1)r(0.8)] {T}="t"^-{{\text{proj}}_2}
:[d(1)r(0.8)] {W~.}^-{t_2},
:[d(1)l(0.8)] {S}="s"_-{{\text{proj}}_1}
(:[d(1)l(0.8)] {U}_-{s_1},
:[d(1)r(0.8)] {V}="v"^-{s_2})
"t" : "v"_{t_1}
"a" :@{}[d(1.7)] |{\textstyle \txt{pb}}
)}
\]
This is well defined since the pull-back is unique up to isomorphism. The identity span
$(1, U, 1): U \ra U$ is defined by
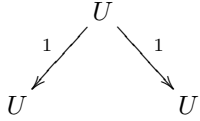
\[
\xygraph{
U="s"
(:[d(1)r(0.9)] {U}^-{1},
:[d(1)l(0.9)] {U}_-{1}
)}
\]
since the composite of it with a span $(s_1, S, s_2): U \ra V$
is given by the following diagram and is equal to the span $(s_1, S, s_2): U \ra V$
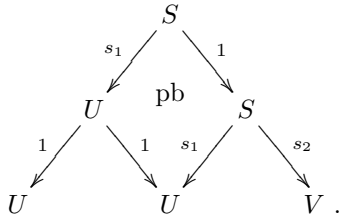
\[\xygraph{
{S}="a"
(:[d(1)r(0.8)] {S}="t"^-{1}
:[d(1)r(0.8)] {V~.}^-{s_2},
:[d(1)l(0.8)] {U}="s"_-{s_1}
(:[d(1)l(0.8)] {U}_-{1},
:[d(1)r(0.8)] {U}="v"^-{1})
"t" : "v"_{s_1}
"a" :@{}[d(1.7)] |{\textstyle \txt{pb}}
)}
\]
This defines the category $\Spn(\E)$. We can write $\Spn(\E)(U,V) \cong [\E / (U \times V)]$
where square brackets denote the isomorphism classes of morphisms.

$\Spn(\E)$ becomes a monoidal category under the tensor product
\[
\xymatrix@1{
\Spn(\E) \times \Spn(\E) \ar[r]^-{\times} & \Spn(\E)}
\]
defined by
\[\xymatrix@1{(U,V) \ar@{|->}[r] & U \times V}\]
\[
[ \xymatrix@1{U \ar[r]^S & U'}, \xymatrix@1{V \ar[r]^T & V'}]
\xymatrix@1{\ar@{|->}[r] &}
[ \xymatrix@1{U \times V \ar[r]^-{S \times T} & U' \times V'}].
\]
This uses the cartesian product in $\E$ yet is not the cartesian product in  $\Spn(\E)$.
It is also compact closed; in fact, we have the following isomorphisms:
$\Spn(\E)(U,V) \cong \Spn(\E)(V, U)$ and $\Spn(\E)(U \times V, W) \cong \Spn(\E)(U, V \times W)$.
The second isomorphism can be shown by the following diagram
\[\vcenter{\xygraph{
S="s"
(:[d(1)r(0.9)] {W},
:[d(1)l(0.9)] {U \times V}
)}}
\quad \lr \quad
\vcenter{\xygraph{
S="s"
(:[d(1)r(0.9)] {W},
:[d(1)l(0.9)] {U},
:[d(1)] {V},
)}}
\quad \lr \quad
\vcenter{\xygraph{
S="s"
(:[d(1)r(0.9)] {V \times W~.},
:[d(1)l(0.9)] {U}
)}}
\]

\section{Direct sums in $\Spn(\E)$} \label{Se3}
Now we assume $\E$ is lextensive.
References for this notion are \cite{Sc}, \cite{CLW}, and \cite{CL}. A category $\E$ is
called lextensive when it has finite limits and finite coproducts such that the functor
\[
\E / A \times \E / B \ra \E / A+B\ ; \quad
\vcenter{\xymatrix{X \ar[d]^f \\ A}} \quad , \quad \vcenter{\xymatrix{Y \ar[d]^g \\ B}} \xymatrix{\ar@{|->}[r] &}
\vcenter{\xymatrix{X+Y \ar[d]^{f+g} \\ A+B}}
\]
is an equivalance of categories for all objects $A$ and $B$. In a lextensive category, coproducts
are disjoint and universal and $0$ is strictly initial. Also we have that the canonical morphism
\[
(A \times B) + (A \times C) \ra A \times (B+C)
\]
is invertible. It follows that $ A \times 0 \cong 0.$

In $\Spn(\E)$ the object $U+V$ is the direct sum of $U$ and $V.$ This can be shown
as follows (where we use lextensivity):
\begin{align*}
\Spn(\E) (U+V, W) & \cong [\E / ((U+V) \times W)] \\
    & \cong [\E / ((U \times W) +( V \times W)) ] \\
    & \simeq [\E / (U \times W)]  \times [\E / (V \times W)]  \\
    &\cong \Spn(\E)(U,W) \times \Spn(\E)(V,W);
\end{align*}
and so $ \Spn(\E)(W, U+V) \cong \Spn(\E)(W,U) \times \Spn(\E)(W,V)$.
Also in the category $\Spn(\E)$, $0$ is the zero object (both initial and terminal):
\[
\Spn(\E)(0, X) \cong [ \E / (0 \times X)] \cong [\E / 0] \cong 1
\]
and so $\Spn(\E)(X,0) \cong 1$.
It follows that $\Spn(\E)$ is a category with homs enriched in commutative monoids.

The addition of two spans $(s_1, S, s_2) : U \ra V$ and $(t_1, T, t_2): U \ra V$
is given by \linebreak $(\nabla \o (s_1 + t_1), S+T, \nabla \o (s_2 + t_2)): U \ra V$ as in
\[\vcenter{\xygraph{
S="s"
(:[d(1)r(0.9)] {V}^-{s_2},
:[d(1)l(0.9)] {U}_-{s_1}
)}}
\quad  + \quad
\vcenter{\xygraph{
T="s"
(:[d(1)r(0.9)] {V}^-{t_2},
:[d(1)l(0.9)] {U}_-{t_1}
)}}
\quad  =
\vcenter{\xygraph{
{S + T}="a"
(:[d(1)r(0.7)] {V+V}="t" ^-{s_2 + t_2}
:[d(1)r(0.7)] {V~.}="v" ^-{\nabla},
:[d(1)l(0.7)] {U+U}="s" _-{s_1 + t_1}
:[d(1)l(0.7)] {U}="u" _-{\nabla})
"a" : @/_6ex/ _-{[s_1,t_1]} "u"
"a" : @/_-6ex/ ^-{[s_2,t_2]} "v"
}}
\]
Summarizing, $\Spn(\E)$ is a monoidal commutative-monoid-enriched category.

There are functors $(-)_*: \E \ra \Spn(\E)$ and $(-)^*: \E^{\text{op}} \ra \Spn(\E)$
which are the identity on objects and take $f: U \ra V$ to $f_*=(1_U, U, f)$ and
$f^*=(f, U, 1_U)$, respectively.

For any two arrows $\xymatrix@1{U \ar[r]^f &V \ar[r]^g & W}$ in $\E$, we have
$(g \o f)_* \cong g_* \o f_*$ as we see from the following diagram
\[
\xygraph{
{U}="a"
(:[d(0.8)r(0.7)] {V}="t"^-{f}
:[d(0.8)r(0.7)] {W~.}^-{g},
:[d(0.8)l(0.7)] {U}="s"_-{1}
(:[d(0.8)l(0.7)] {U}_-{1},
:[d(0.8)r(0.7)] {V}="v"^-{f})
"t" : "v"_{1}
"a" :@{}[d(1.5)] |{\textstyle \txt{pb}}
)}
\]
Similarly $(g \o f)^* \cong f^* \o g^*$.

\section{Mackey functors on $\E$} \label{Se4}

A \emph{Mackey functor} $M$ from $\E$ to the category $\kMod$ of $k$-modules consists
of two functors
\[ M_*: \E \ra \kMod, \quad M^*: \E^{\text{op}} \ra \kMod \]
such that:
\begin{enumerate}
\item $M_*(U) = M^*(U) \quad (=M(U))$ for all $U$ in $\E$
\item
for all pullbacks
\[\xygraph{
P="a"
(:[r(1.8)] {V}^-{q}
:[d(1.4)] {W}="t"^-{s},
:[d(1.4)] {U}_-{p}
:"t"_-{r}
)}\]
in $\E$, the square (which we call a \emph{Mackey square})
\[\xygraph{
{M(U)}="a"
(:[r(1.8)] {M(W)}_-{M_*(r)}
:[u(1.4)] {M(V)}="t"_-{M^*(s)},
:[u(1.4)] {M(P)}^-{M^*(p)}
:"t"^-{M_*(q)}
)}\]
commutes, and
\item for all coproduct diagrams
\[
\xymatrix@1{
U \ar[r]^-{i} & {U+V} & V \ar[l]_-{j} }
\]
in $\E$, the diagram
\[
\xymatrix@C=7ex{
{M(U)} \ar@<-0.8ex>[r]_-{M_*(i)}
& {M(U+V)} \ar@<-1.1ex>[l]_-{M^*(i)} \ar@<1.1ex>[r]^-{M^*(j)}
& {M(V)} \ar@<1.0ex>[l]^-{M_*(j)} }
\]
is a direct sum situation in \kMod. (This implies $M(U+V) \cong M(U)\oplus M(V)$.)
\end{enumerate}

A morphism $\theta : M \ra N$ of Mackey functors is a family of morphisms
$\theta_U : M(U) \ra$ $ N(U)$ for $U$ in $\E$ which defines natural transformations
$\theta_* : M_* \ra N_*$ and $\theta^* : M^* \ra N^*$.

\begin{proposition} (Lindner \cite{Li1})
The category $\Mky(\E,\kMod)$ of Mackey functors, from a lextensive category $\E$ to
the category $\kMod$ of $k$-modules, is equivalent to $[\Spn(\E), \kMod]_+$, the category
of coproduct-preserving functors.
\end{proposition}

\begin{proof}
Let $M$ be a Mackey functor from $\E$ to $\kMod$. Then we have a pair $(M_*, M^*)$
such that $M_*: \E \ra \kMod$, $M^*: \E^{\text{op}} \ra \kMod$ and $M(U) = M_*(U) = M^* (U)$.
Now define a functor $M: \Spn(\E) \ra \kMod$ by $M(U) = M_*(U) = M^*(U)$ and
\[ M \left ( \vcenter{\xygraph{
 S="s"
(:[d(1)r(0.8)] {V}^-{s_2},
:[d(1)l(0.8)] {U}_-{s_1}
)  }}  \right )
\quad  = \quad
\big ( \xygraph{
{M(U)}
(:[r(1.8)]{M(S)}^-{M^*(s_1)}
:[r(1.8)]{M(V)}^-{M_*(s_2)}
)}
\big ).
\]
We need to see that $M$ is well-defined. If $h: S \ra S'$ is an
isomorphism, then the following diagram
\[\xygraph{
{S'}="a"
(:[r(2)] {S'}^-{1}
:[d(1.2)] {S'}="t"^-{1},
:[d(1.2)] {S}_-{h^{-1}}
:"t"_-{h}
)}
\]
is a pull back diagram. Therefore $M^*(h^{-1})  = M_*(h)$ and $M_*(h^{-1})  = M^*(h)$.
This gives, $M_*(h)^{-1} = M^*(h)$. So if $h:(s_1,S,s_2) \ra (s'_1,S',s'_2)$ is an isomorphism
of spans, we have the following commutative diagram.
\[
\xymatrix @=9ex@R=11ex@C=11ex@ur{
{M(U)} \ar[r]^-{M^*(s_1)}  \ar[d]_-{M^*(s'_1)} & {M(S)} \ar[d]^-{M_*(s_2)}
\ar@<1ex>[dl]^-{M_*(h)} \\
{M(S')} \ar[r]_-{M_*(s'_2)} \ar@<1ex>[ur]^-{M^*(h)} & {M(V)} }
\]
Therefore we get
\[
M_*(s_2) M^*(s_1) = M_*(s'_2)M^*(s'_1).
\]
From this definition $M$ becomes a functor, since
\begin{equation*}
\begin{split} \hspace{-5pt}
 M \left ( \vcenter{\xygraph{
{P}="a"
(:[d(1)r(0.9)] {T}="t"^-{p_2}
:[d(1)r(0.9)] {W}^-{t_2},
:[d(1)l(0.9)] {S}="s"_-{p_1}
(:[d(1)l(0.9)] {U}_-{s_1},
:[d(1)r(0.9)] {V}="v"^-{s_2})
"t" : "v"_{t_1}
"a" :@{}[d(2)] |{\textstyle \txt{pb}}
)}}  \right )
& =
\vcenter{\xygraph{
{M(U)}="u"
(:[r(1.7)]{M(P)}="p"^-{M^*(p_1s_1)}
:[r(1.7)]{M(W)}="w"^-{M_*(t_2p_2)},
:[d(1.1)r(0.9)]{M(S)}="s"_-{M^*(s_1)}
:[d(1.1)r(0.9)] {M(V)}="v"_-{M_*(s_2)}
:[u(1.1)r(0.9)] {M(T)}="t"_-{M^*(t_1)}
:"w"_-{M_*(t_2)}
"s":"p"^-{M^*(p_1)}
"p":"t"^{M_*(p_2)}
"p" :@{}[d(2.3)r(0.3)] |{\textstyle \txt{Mackey}}
)}}  \\
& =
( \xygraph{
{M(U)}
(:[r(1.7)]{M(V)}^-{M(s_1, S, s_2)}
:[r(1.7)]{M(W)}^-{M(t_1, T, t_2)}
)} ),
\end{split}
\end{equation*}
where $P=S \times_V T$ and $p_1$ and $p_2$ are the projections 1 and 2 respectively,
so that
\[
M((t_1, T, t_2)  \o (s_1, S, s_2)) = M(t_1, T, t_2) \o M(s_1, S, s_2).
\]
The value of $M$ at the identity span $(1, U, 1): U \ra U$ is given by
\begin{equation*}
\begin{split}
M \left ( \vcenter{\xygraph{
{U}
(:[d(1)r(0.9)] {U}^-{1},
:[d(1)l(0.9)] {U}_-{1}
)}} \right )
& \quad =  \quad
( \xygraph{
{M(U)}
(:[r(1.5)]{M(U)}^-{1}
:[r(1.5)]{M(U)}^-{1}
)}
) \\
& \quad = \quad
( 1 : \xygraph{
{M(U)}
(:[r(1.5)]{M(U)}
)}
).
\end{split}
\end{equation*}

Condition (3) on the Mackey functor clearly is equivalent to the requirement that
$M: \Spn(\E) \ra \kMod$ should preserve coproducts.

Conversely, let $M: \Spn(\E) \ra \kMod$ be a functor. Then we can define two
functors $M_*$ and $ M^*$, referring to the diagram
\[
\xygraph{
{\E}="a"
(:[r(1.7)]{\Spn(\E)}="b"^-{(-)_*}
:[r(2.2)]{\kMod ~,}^-{M},
[d(0.8)]{\E^{\text{op}}}="c"
"c":"b"_-{(-)^*}
)}
\]
by putting $M_* = M \o (-)_*$ and $M^* = M \o (-)^*$.
The Mackey square is obtained by using the functoriality of $M$ on the composite
span
\[
s^* \o r_* = (p, P, q) = q_* \o p^* .
\]
The remaining details are routine.
\end{proof}


\section{Tensor product of Mackey functors} \label{Se5}
We now work with a general compact closed category $\T$ with finite products. It
follows (see \cite{Ho}) that $\T$ has direct sums and therefore that $\T$ is enriched in the
monoidal category $\V$ of commutative monoids. We write $\ox$ for the tensor product in
$\T$, write $I$ for the unit, and write $(-)^*$ for the dual. The main example we have in mind
is $\Spn(\E)$ as in the last section where $\ox = \times, I=1,$ and $V^*=V$.
A Mackey functor on $\T$ is an additive functor $M: \T \ra \kMod$.

Let us review the monoidal structure on the category $\V$ of commutative monoids; the binary
operation of the monoids will be written additively. It is monoidal closed. For $A, B \in \V,$ the commutative monoid
\[
[A,B] = \lbrace f: A \ra B ~ | ~ f ~\text{is a commutative monoid morphism} \rbrace ,
\]
with pointwise addition, provides the internal hom and there is a tensor product $A \ox B$ satisfying
\[
\V(A\ox B, C) \cong \V(A, [B,C]).
\]
The construction of the tensor product is as follows.
The free commutative monoid $FS$ on a set $S$ is
\[
FS = \lbrace u: S \ra \mathbb{N} ~ | ~ u(s)=0 ~ \text{for all but a finite number
of} ~ s \in S \rbrace \subseteq \mathbb{N}^{S}.
\]
For any $A,B \in \V$,
\begin{equation*}
A \ox B  = \left (
\begin{aligned}
 F(A \times B) / & (a+a', b) \sim (a,b) + (a',b) \\
            & (a, b+b') \sim (a,b) + (a,b')
\end{aligned}
\right ).
\end{equation*}

We regard $\T$ and $\kMod$ as $\V$-categories. Every $\V$-functor
$\T \ra \kMod$ preserves finite direct sums. So $[\T, \kMod]_+$
is the $\V$-category of $\V$-functors.

For each $V \in \V$ and $X$ an object of a $\V$-category $\X$, we write $V \ox X$ for the object
(when it exists) satisfying
\[
\X(V \ox X, Y) \cong [V, \X(X,Y)]
\]
$\V$-naturally in $Y$. Also the coend we use is in the $\V$-enriched sense:
for the functor $T: \C^{\text{op}} \ox \C \ra \X$, we have a coequalizer
\[
\xymatrix@C=5ex{
{\displaystyle \sum_{V, W} \C(V,W) \ox T(W,V)}
\ar@<0ex>[r] \ar@<1.7ex>[r]
& {\displaystyle \sum_V T(V, V)} \ar@<0.5ex>[r]
& {\displaystyle \int^V T(V,V)}}
\]
when the coproducts and tensors exist.

The tensor product of Mackey functors can be defined by convolution (in the sense of \cite{Da}) in
$[\T, \kMod]_+$ since $\T$ is a monoidal category. For Mackey functors
$M$ and $N$, the tensor product $M*N$ can be written as follows:
\begin{equation*}
\begin{split}
(M * N)(Z) & = \int^{X,Y} \T (X \ox Y, Z) \ox M(X) \ox_k N(Y) \\
& \cong \int^{X,Y} \T (Y, X^* \ox Z) \ox M(X) \ox_k N(Y) \\
& \cong \int^{X} M(X) \ox_k N(X^* \ox Z) \\
&  \cong \int^{Y} M(Z\ox Y^*) \ox_k N(Y) .
\end{split}
\end{equation*}
the last two isomorphisms are given by the Yoneda lemma.

The \emph{Burnside functor} $J$ is defined to be the Mackey functor on $\T$ taking an
object $U$ of $\T$ to the free $k$-module on $\T(I,U)$. The Burnside functor is the unit
for the tensor product of the category $\Mky$.

This convolution satisfies the necessary
commutative and associative conditions for a symmetric monoidal category (see \cite{Da}).
$[\T, \kMod]_+$ is also an abelian category (see \cite{Fr}).

When $\T$ and $k$ are understood, we simply write $\Mky$ for this category
$[\T, \kMod]_+$.

\section{The Hom functor} \label{Se6}

We now make explicit the closed structure on $\Mky$. The Hom Mackey functor
is defined by taking its value at the Mackey functors $M$ and $N$ to be
\[
\Hom (M, N)(V) = \Mky(M(V^* \ox -), N),
\]
functorially in $V$. To see that this hom has the usual universal property with respect to tensor,
notice that we have the natural bijections below (represented by horizontal
lines).
\[
\infer{L(V) \ra \Mky (M(V^* \ox -), N) ~ ~
    \text{natural in $V$}}{
\infer{L(V) \ra {\displaystyle\int_U \text{Hom}_k (M(V^* \ox U), N(U)) ~ ~
    \text{natural in $V$}}}{
\infer{L(V) \ra \Hom_k (M(V^* \ox U), N(U))~ ~
    \text{dinatural in $U$ and natural in $V$}}{
\infer{L(V) \ox_k M(V^* \ox U) \ra N(U) ~ ~
    \text{natural in $U$ and dinatural in $V$}}
{(L*M)(U) \ra  N(U) ~ ~ \text{natural in~} U }}}}
\]
We can obtain another expression for the hom using the isomorphism
\[
\T(V \ox U, W) \cong \T(U, V^* \ox W)
\]
which shows that we have adjoint functors
\[
\xymatrix@C=8ex{\T \ar@{}[r]|{\bot} \ar@<0.3ex>@/^/ [r] ^{V \ox -} &
\T ~. \ar@<0.3ex>@/^/ [l] ^{V^* \ox -}}
\]
Since $M$ and $N$ are Mackey functors on $\T,$ we obtain a diagram
\[
\xymatrix@C=2ex{\T \ar@{}[rr]|{\bot} \ar@<0.3ex>@/^/ [rr] ^{V \ox -} \ar[dr]_N &&
\T \ar@<0.3ex>@/^/ [ll] ^{V^* \ox -} \ar[dl]^M \\
& \kMod}
\]
and an equivalence of natural transformations
\[
\infer{M(V^* \ox -) \Longrightarrow  N}{M \Longrightarrow N(V \ox -)}.
\]
Therefore, the Hom Mackey functor is also given by
\[
\Hom (M, N)(V) = \Mky (M, N(V \ox -)).
\]

\section{Green functors} \label{Se7}
A \emph{Green functor} $A$ on $\T$ is a Mackey functor (that is, a coproduct preserving functor
$A:  \T \ra \kMod$) equipped with a monoidal structure made up of a natural transformation
\[
\mu: A(U) \ox_k A(V) \ra A(U \ox V),
\]
for which we use the notation $\mu(a \ox b)= a.b$ for $a\in A(U)$, $b \in A(V)$, and
a morphism
\[
 \eta: k \ra A(1),
\]
whose value at $1 \in k$ we denote by 1.
Green functors are the monoids in $\Mky$. If $A, B: \T \ra \kMod$
are Green functors then we have an isomorphism
\[
\Mky(A*B,C) \cong \text{Nat}_{U,V} (A(U) \ox_k B(V), C(U \ox V) ).
\]
Referring to the square
\[\xygraph{
{\T \ox \T}="a"
(:[r(3.2)] {\kMod \ox \kMod}^-{A \ox B}
:[d(1.2)] {\kMod~~~,}="t"^-{\ox_k},
:[d(1.2)] {\T}_-{\ox}
:"t"_-{C}
)}
\]
we write this more precisely as
\[
\Mky(A*B,C) \cong [\T \ox \T, \kMod](\ox_k \o (A \ox B), C \o \ox).
\]
The Burnside functor $J$ and $\text{Hom}(A, A)$ (for any Mackey functor $A$) are monoids in $\Mky$ and so are
Green functors.

We denote by $\Grn(\T, \kMod)$ the category of Green functors on $\T$. When $\T$ and $k$ are
understood, we simply write this as $\Grn (=\mathbf{Mon}(\Mky))$ consisting of the monoids in $\Mky$.

\section{Dress construction} \label{Se8}
The Dress construction (\cite{Bo2}, \cite{Bo3}) provides a family of endofunctors
$D(Y,-)$ of the category $\Mky$, indexed by the objects $Y$ of $\T$. The Mackey functor
defined as the composite
\[
\xymatrix@1{
\T \ar[r]^{-\ox Y} & \T \ar[r]^M & \Mod_k}
\]
is denoted by $M_Y$ for $M\in \Mky$; so $M_Y(U)=M(U \ox Y)$.
We then define the \emph{Dress construction}
 \[
D: \T \ox \Mky \ra \Mky
\]
by $D(Y,M)=M_Y$. The $\V$-category $\T \ox \Mky$ is monoidal via the pointwise
structure:
\[
(X,M) \ox (Y,N) = (X\ox Y, M*N).
\]

\begin{proposition} \label{Pro8.1}
The Dress construction
\[
D: \T \ox \Mky \ra \Mky
\]
is a strong monoidal $\V$-functor.
\end{proposition}

\begin{proof}
We need to show that $D((X,M) \ox (Y,N)) \cong D(X,M)*D(Y,N)$; that is,
$M_X*M_Y \cong (M*N)_{X \ox Y}$.
For this we have the calculation
\begin{equation*}
\begin{split}
(M_X*N_Y)(Z) & \cong \int^U M_X(U) \ox_k N_Y(U^* \ox Z) \\
    & = \int^U M(U \ox X) \ox_k N(U^* \ox Z \ox Y) \\
    & \cong \int^{U,V} \T(V,U \ox X) \ox M(V) \ox_k N(U^* \ox Z \ox Y) \\
    & \cong \int^{U,V} \T(V \ox X^*,U) \ox M(V) \ox_k N(U^* \ox Z \ox Y) \\
    & \cong \int^V M(V) \ox_k N(V^* \ox X \ox Z \ox Y) \\
    & \cong (M*N)(Z \ox X \ox Y) \\
    & \cong (M*N)_{X \ox Y}(Z).
\end{split}
\end{equation*}
Clearly we have $D(I,J) \cong J$. The coherence conditions are readily checked.
\end{proof}
We shall analyse this situation more fully in Remark \ref{Re8.5} below.

We are interested, after \cite{Bo2}, in when the Dress construction induces a family of endofunctors
on the category $\Grn$ of Green functors. That is to say, when is there a natural structure of Green functor
on $A_Y=D(Y,A)$ if $A$ is a Green functor? Since $A_Y$ is the composite
\[
\xymatrix@1{
\T \ar[r]^{-\ox Y} & \T \ar[r]^A & \Mod_k}
\]
with $A$ monoidal, what we require is that $-\ox Y$ should be monoidal (since monoidal
functors compose). For this we use Theorem 3.7 of \cite{DPS}:

\emph{ if $Y$ is a monoid in the lax centre $\Z_l(\T)$ of $\T$ then $-\ox Y: \T \ra \T$
is canonically monoidal.}

Let $\C$ be a monoidal category. The lax centre $\Z_l(\C)$ of $\C$ is defined to have objects the pairs
$(A, u)$ where $A$ is an object of $\C$ and $u$ is a natural family of morphisms
$u_B: A \ox B \ra B \ox A$ such that the following two diagrams commute
\[
\vcenter{\xygraph{
{A \ox B \ox C}="s"
(:[r(3.2)] {B \ox C \ox A}="t"^-{u_{B \ox C}},
:[d(1.2)r(1.6)] {B \ox A \ox C}_-{u_B \ox 1_C}
:"t" _-{1_B \ox u_C}
)}}
\qquad
\vcenter{\xygraph{
{A \ox I}="s"
(:[r(3.2)] {I \ox A}^-{u_I}
:[d(1.2)l(1.6)] {A ~.}="t"^-{\cong},
:"t" _-{\cong}
)}}
\]
Morphisms of $\Z_l(\C)$ are morphisms in $\C$ compatible with the $u$. The tensor product is
defined by
\[
(A,u) \ox (B,v) = (A \ox B, w)
\]
where $w_C=(u_C \ox 1_B) \o (1_A \ox v_C)$.
The centre $\Z(\C)$ of $\C$ consists of the objects $(A,u)$ of $\Z_l(\C)$ with each $u_B$
 invertible.

It is pointed out in \cite{DPS} that, when $\C$ is cartesian monoidal, an
object of $\Z_l(\C)$ is merely an object $A$ of $\C$ together with a natural family
$A \times X \ra X$. Then we have the natural bijections below (represented by horizontal lines)
for $\C$ cartesian closed:
\[
\infer{A \ra \displaystyle{\int_X [X, X] } ~ ~ \text{in} ~ \C.}{
\infer{A \ra [X, X] ~ ~ \text{dinatural in}~ X}
      {A \times X \ra X ~ ~ \text{natural in}~ X}}
\]
Therefore we obtain an equivalence
$\Z_l(\C) \simeq \C / \int_X [X, X] .$

The internal hom in $\E$, the category of finite $G$-sets for the finite group $G$,
is $[X,Y]$ which is the set of functions $r: X \ra Y$ with
$(g.r)(x)=gr(g^{-1}x)$. The $G$-set $\int_X [X, X]$ is defined by
\[
\int_X [X, X] = \biggl \lbrace r=(r_X: X \longrightarrow  X)  \quad \Big \rvert~ f \o r_X = r_Y \o f~
\text{for all}~ G \text{-maps}~ f: X \longrightarrow Y  \biggr \rbrace
 \]
with $(g.r)_X(x)=gr_X(g^{-1}x)$.

\begin{lemma}\label{5.1}
The $G$-set $\displaystyle \int_X [X, X]$ is isomorphic to $G_c$, which is
the set $G$ made a $G$-set by conjugation action.
\end{lemma}

\begin{proof}
Take $r \in \int_X [X,X]$. Then we have the commutative square
\[\xygraph{
{G}="a"
(:[r(2.2)] {G}^-{r_G}
:[d(1.2)] {X}="t"^-{\hat{x}},
:[d(1.2)] {X}_-{\hat{x}}
:"t"_-{r_X}
)}
\]
where $\hat{x}(g)=gx$ for $x \in X$.
So we see that $r_X$ is determined by $r_G(1)$ and
\begin{equation*}
\begin{split}
(g.r)_G(1) & = gr_G(g^{-1}1) \\
    & =  gr_G(g^{-1})\\
    & = gr_G(1)g^{-1}.
\end{split}
\end{equation*}
\end{proof}

As a consequence of this Lemma, we have $\Z_l(\E) \simeq \E / G_c$
where $\E / G_c$ is the category of crossed $G$-sets of Freyd-Yetter (\cite{FY1}, \cite{FY2})
who showed that $\E / G_c$ is a braided monoidal category. Objects are pairs $(X, |~|)$
where $X$ is a $G$-set
and $|~|: X \ra G_c$ is a $G$-set morphism (``equivariant function'')
meaning $|gx|=g|x|g^{-1}$ for $g \in G$, $x \in X$.
The morphisms $f: (X, |~|) \ra (Y, |~|)$ are functions $f$ such that the following diagram commutes.
\[ \xygraph{
{X}="s"
(:[r(2.8)] {Y}^-{f}
:[d(1.2)l(1.4)] {G_c}="t"^-{|~|},
:"t" _-{|~|}
)}
\]
That is, $f(gx)=gf(x)$.

Tensor product is defined by
\[
(X,|~|) \ox (Y,|~|) = (X \times Y, \Vert ~ \Vert),
\]
where $\Vert (x,y) \Vert = |x| |y|$.

\begin{proposition}\label{5.2} \cite[Theorem 4.5]{DPS}
The centre $\Z(\E)$ of the category $\E$ is equivalent to the category $\E / G_c$ of crossed $G$-sets.
\end{proposition}

\begin{proof}
We have a fully faithful functor $\Z(\E) \ra \Z_l(\E)$ and so $\Z(\E) \ra \E / G_c$.
On the other hand, let $|~|: A \ra G_c$ be an object of $\E / G_c$; so
$|ga|g=g|a|$. Then the corresponding object of $\Z_l(\E)$ is $(A,u)$ where
\[
u_X: A \times X \ra X \times A
\]
with
\[
u_X(a,x) = (|a|x, a).
\]
However this $u$ is invertible since we see that
\[
{u_X}^{-1}(x, a) = (a, |a|^{-1}x).
\]
This proves the proposition.
\end{proof}

\begin{theorem} \cite[{Bo2}] {Bo3}
If $Y$ is a monoid in $\E / G_c$ and $A$ is a Green functor for $\E$ over $k$ then
$A_{Y}$ is a Green functor for $\E$ over $k$, where $A_{Y}(X)=A(X \times Y)$.
\end{theorem}

\begin{proof}
We have $\Z(\E) \simeq \E / G_c $, so $Y$ is a monoid
in $\Z(\E)$. This implies $- \times Y : \E \ra \E$ is a monoidal functor (see Theorem 3.7
of \cite{DPS}). It also preserves pullbacks. So $- \times Y : \Spn(\E) \ra \Spn(\E)$ is a
monoidal functor . If $A$ is a
Green functor for $\E$ over $k$ then $A: \Spn(\E) \ra \kMod$ is monoidal. Then we get
$A_{Y} = A \o (- \times Y): \Spn(\E) \ra \kMod$ is monoidal and
$A_{Y}$ is indeed a Green functor for $\E$ over $k$.
\end{proof}

\begin{remark} \label{Re8.5}
The reader may have noted that Proposition \ref{Pro8.1} implies that $D$ takes monoids to monoids.
A monoid in $\T \ox \Mky$ is a pair $(Y,A)$ where $Y$ is a monoid in $\T$ and $A$ is a Green functor;
so in this case, we have that $A_Y$ is a Green functor. A monoid $Y$ in $\E$ is certainly a monoid in $\T$.
Since $\E$ is cartesian monoidal (and so symmetric), each monoid in $\E$ gives one in the centre.
However, not every monoid in the centre arises in this way. The full result behind Proposition
\ref{Pro8.1} and the centre situation is: the Dress construction
\[
D : \Z(\T) \ox \Mky \ra \Mky
\]
is a strong monoidal $\V$ -functor; it is merely monoidal when the centre is replaced by the lax centre.

It follows that $A_Y$ is a Green functor whenever $A$ is a Green functor and $Y$ is a monoid in the lax
centre of $\T$.
\end{remark}


\section{Finite dimensional Mackey functors} \label{Se9}

We make the following further assumptions on the symmetric compact closed category $\T$ with
finite direct sums:
\begin{itemize}
\item there is a finite set $\C$ of objects of $\T$ such that every object $X$ of $\T$ can be
written as a direct sum
\[
X \cong  \bigoplus_{i=1}^{n} C_i
\]
with $C_i$ in $\C$; and

\item each hom $\T(X,Y)$ is a finitely generated commutative monoid.
\end{itemize}

Notice that these assumptions hold when $\T=\Spn(\E)$ where $\E$ is the category of
finite $G$-sets for a finite group $G$. In this case we can take $\C$ to consist of a representative set of connected (transitive) $G$-sets. Moreover, the spans $S: X \ra Y$ with $S \in \C$ generate the
monoid $\T(X,Y)$.

We also fix $k$ to be a field and write $\Vect$ in place of $\Mod_k$.

A Mackey functor $M: \T \ra \Vect$ is called \emph{finite dimensional} when each $M(X)$
is a finite-dimensional vector space. Write $\Mkyf$ for the full subcategory of $\Mky$ consisting of these.

We regard $\C$ as a full subcategory of $\T$. The inclusion functor $\C \ra \T$ is dense and the
density colimit presentation is preserved by all additive $M: \T \ra \Vect$. This is shown as follows:

\begin{align*}
\int^C \T(C,X) \ox M(C) & \cong \int^C \T(C,\bigoplus_{i=1}^{n} C_i) \ox M(C) \\
    & \cong \bigoplus_{i=1}^{n} \int^C \T(C,C_i) \ox M(C) \\
    & \cong \bigoplus_{i=1}^{n} \int^C \C(C,C_i) \ox M(C) \\
    & \cong \bigoplus_{i=1}^{n} M(C_i) \\
    & \cong  M(\bigoplus_{i=1}^{n} C_i) \\
    &\cong M(X).
\end{align*}
That is,
\[
M \cong \int^C \T(C,-) \ox M(C).
\]

\begin{proposition}
The tensor product of finite-dimensional Mackey functors is
finite dimensional.
\end{proposition}

\begin{proof}
Using the last isomorphism, we have
\begin{equation*}
\begin{split}
(M * N)(Z) & = \int^{X,Y} \T(X \ox Y, Z) \ox M(X) \ox_k N(Y) \\
& \cong \int^{X,Y,C,D} \T(X \ox Y,Z) \ox \T(C,X) \ox \T(D,Y) \ox M(C) \ox_k N(D) \\
&  \cong \int^{C,D} \T(C \ox D,Z) \ox M(C) \ox_k N(D) .
\end{split}
\end{equation*}

If $M$ and $N$ are finite dimensional then so is the vector space $\T(C\ox D,Z)\ox M(C) \ox_k N(D)$
(since $\T(C \ox D,Z)$ is finitely generated). Also the coend is a quotient of a finite direct sum.
So $M*N$ is finite dimensional.
\end{proof}

It follows that $\Mkyf$ is a monoidal subcategory of $\Mky$ (since the Burnside functor $J$ is
certainly finite dimensional under our assumptions on $\T$).

The promonoidal structure on $\Mkyf$ represented by this monoidal structure can be expressed in many ways:
\begin{equation*}
\begin{split}
P(M,N;L) & = \Mkyf(M*N, L) \\
    & \cong \text{Nat}_{X,Y,Z} (\T(X \ox Y, Z) \ox M(X) \ox_k N(Y),L(Z)) \\
    & \cong \text{Nat}_{X,Y} (M(X) \ox_k N(Y),L(X \ox Y)) \\
    & \cong \text{Nat}_{X,Z} (M(X) \ox_k N(X^* \ox Z),L(Z)) \\
    & \cong \text{Nat}_{Y,Z} (M(Z \ox Y^*) \ox_k N(Y),L(Z)).
\end{split}
\end{equation*}

Following the terminology of \cite{DS}, we say that a promonoidal category $\M$ is \emph{$*$-autonomous}
when it is equipped with an equivalence $S: \M^{\text{op}} \ra \M$ and a natural isomorphism
\[
P(M, N; S(L)) \cong P(N, L; S^{-1}(M)).
\]
A monoidal category is $*$-autonomous when the associated promonoidal category is.

As an application of the work of Day \cite{Da4} we obtain that $\Mkyf$ is $*$-autonomous. We shall
give the details.

For $M \in \Mkyf$, define $S(M)(X)=M(X^*)^*$ so that $S: \Mkyf^{\text{op}} \ra \Mkyf$ is its own inverse
equivalence.

\begin{theorem}
The monoidal category $\Mkyf$ of finite-dimensional Mackey functors on $\T$
is $*$-autonomous.
\end{theorem}

\begin{proof}
With $S$ defined as above, we have the calculation:
\begin{equation*}
\begin{split}
P(M,N;S(L)) & \cong \text{Nat}_{X,Y} (M(X) \ox_k N(Y),L(X^* \ox Y^*)^*) \\
    & \cong \text{Nat}_{X,Y} (N(Y) \ox_k L(X^* \ox Y^*), M(X)^*) \\
    & \cong \text{Nat}_{Z,Y} (N(Y) \ox_k L(Z \ox Y^*), M(Z^*)^*) \\
    & \cong \text{Nat}_{Z,Y} (N(Y) \ox_k L(Z \ox Y^*), S(M)(Z)) \\
    &\cong P(N,L;S(M)).
\end{split}
\end{equation*}
\end{proof}

\section{Cohomological Mackey functors} \label{Se10}

Let $k$ be a field and $G$ be a finite group. We are interested in the relationship between ordinary
$k$-linear representations of $G$ and Mackey functors on $G$.

Write $\E$ for the category of finite $G$-sets as usual. Write $\r$ for the category $\Rep_k(G)$
of finite -dimensional $k$-linear representations of $G$.

Each $G$-set $X$ determines a $k$-linear representation $kX$ of $G$ by extending the action of $G$
linearly on $X$. This gives a functor
\[
k: \E \ra \r.
\]
We extend this to a functor
\[
k_* : \T^{\text{op}} \ra \r,
\]
where $\T= \Spn(\E)$, as follows. On objects $X \in \T$, define
\[
k_*X=kX.
\]
For a span $(u,S,v): X \ra Y$ in $\E$, the linear function $k_*(S):kY \ra kX$ is defined by
\[
k_*(S)(y) = \sum_{v(s)=y} u(s) ~;
\]
this preserves the $G$-actions since
\[
k_*(S)(gy) = \sum_{v(s)=gy} u(s) = \sum_{v(g^{-1}s)=y} gu(g^{-1}s) ~ = ~ gk_*(S)(y).
\]
Clearly $k_*$ preserves coproducts.

By the usual argument (going back to Kan, and the geometric realization and singular functor adjunction),
we obtain a functor
\[
\widetilde{k_*}: \r \ra \Mky(G)_{\textit{fin}}
\]
defined by
\[
\widetilde{k_*}(R)= \r(k_*-,R)
\]
which we shall write as $R^{-}: \T \ra \Vect_k$. So
\[
R^X= \r(k_*X,R) \cong \GSet(X,R)
\]
with the effect on the span $(u,S,v): X \ra Y$ transporting to the linear function
\[
\GSet(X,R) \ra \GSet(Y,R)
\]
which takes $\tau: X \ra R$ to $\tau_S: Y \ra R$ where
\[
\tau_S(y) = \sum_{v(s)=y} \tau(u(s)).
\]

The functor $\widetilde{k_*}$ has a left adjoint
\[
\textit{colim}(-,k_*): \Mky(G)_{\textit{fin}} \ra \r
\]
defined by
\[
\textit{colim}(M,k_*) = \int^C M(C) \ox_k k_*C
\]
where $C$ runs over a full subcategory $\C$ of $\T$ consisting of a representative set of connected
$G$-sets.

\begin{proposition}
The functor $\widetilde{k_*}: \Rep_k(G) \ra \Mky(G)$ is fully faithful.
\end{proposition}

\begin{proof}
For $R_1, R_2 \in \r$, a morphism $\theta: R_1^- \ra R_2^-$ in $\Mky(G)$ is a family of linear functions
$\theta_X$ such that the following square commutes for all spans $(u,S,v): X \ra Y$ in $\E$.

\[\xygraph{
{\GSet(X,R_1)}
(:[r(3)] {\GSet(X,R_2)}^-{\theta_X}
:[d(1.4)] {\GSet(Y,R_2)}="t"^-{(~-~)_S},
:[d(1.4)] {\GSet(Y,R_1)}_-{(~-~)_S}
:"t"_-{\theta_Y}
)}
\]
Since $G$ (with multiplication action) forms a full dense subcategory of $\GSet$, it follows that we obtain a unique morphism $f: R_1 \ra R_2$ in $\GSet$ such that
\[
f(r)= \theta_G(\hat{r})(1)
\]
(where $\hat{r}: G \ra R$ is defined by $\hat{r}(g)=gr$ for $r \in R$); this is a special case of Yoneda's Lemma.
Clearly $f$ is linear since $\theta_G$ is. By taking $Y=G, S=G$ and $v=1_G: G \ra G$, commutativity of the above square yields
\[
\theta_X(\tau)(x) = f(\tau(x));
\]
that is, $\theta_X = \widetilde{k_*}(f)_X$.
\end{proof}

An important property of Mackey functors in the image of $\widetilde{k_*}$ is that they are
\emph{cohomological} in the sense of \cite{We}, \cite{Bo4} and \cite{TW}. First we recall some classical terminology associated with a Mackey functor $M$ on a group $G$.

For subgroups $K \leq H$ of $G$, we have the canonical $G$-set morphism $\sigma^H_K: G/K \ra G/H$
defined on the connected $G$-sets of left cosets by $\sigma^H_K(gK)=gH$. The linear functions
\begin{gather*}
r^H_K=M_*(\sigma^H_K): M(G/H) \ra M(G/K) \quad \text{and} \\
t^H_K=M^*(\sigma^H_K): M(G/K) \ra M(G/H) \qquad \quad
\end{gather*}
are called \emph{restriction} and \emph{transfer} (or \emph{trace} or \emph{induction}).

A Mackey functor $M$ on $G$ is called \emph{cohomological} when each composite
$t^H_K r^H_K: M(G/H)$ $ \ra M(G/H)$ is equal to multiplication by the index $[H:K]$ of $K$ in $H$.
We supply a proof of the following known example.

\begin{proposition}
For each $k$-linear representation $R$ of $G$, the Mackey functor
$\widetilde{k_*}(R)=R^-$ is cohomological.
\end{proposition}

\begin{proof}
With $M=R^-$ and $\sigma=\sigma^H_K$, notice that the function
\[
t^H_Kr^H_K = M^*(\sigma)M_*(\sigma) = M(\sigma, G/K,1)M(1,G/K, \sigma)
=M(\sigma, G/K, \sigma)
\]
takes $\tau \in \E(G/H,R)$ to $\tau_{G/K} \in \E(G/H,R)$ where
\[
\tau_{G/K}(H) = \sum_{\sigma(s)=H} \tau(\sigma(s)) = \sum_{\sigma(s)=H} \tau(H)
= \quad (\sum_{\sigma(s)=H} 1) \tau(H)
\]
and $s$ runs over the distinct $gK$ with $\sigma(s)=gH=H$;  the number of distinct $gK$
with $g \in H$ is of course $[H:K]$. So $\tau_{G/K}(xH)=[H:K] \tau(xH)$.
\end{proof}

\begin{lemma} \label{10.3}
The functor $k_*: \T^{\text{op}} \ra \r$ is strong monoidal.
\end{lemma}

\begin{proof}
Clearly the canonical isomorphisms
\[
k(X_1 \times X_2) \cong kX_1 \ox kX_2, \quad k1\cong k
\]
show that $k:\E  \ra \r$ is strong monoidal. All that remains to be seen is that these isomorphisms are natural
with respect to spans $(u_1,S_1,v_1):X_1 \ra Y_1, (u_2,S_2,v_2):X_2 \ra Y_2$. This comes down to the bilinearity of tensor product:
\[
\underset{v_2(s_2)=y_2}{\underset{v_1(s_1)=y_1}{\sum}} u_1(s_1)\ox u_2(s_2) =
\sum_{v_1(s_1)=y_1} u_1(y_1) \ox \sum_{v_2(s_2)=y_2} u_2(y_2).
\]
\end{proof}

We can now see that the adjunction
\[
\xymatrix@C=13pt{\text{\emph{colim}}(-,k_*)  & \ar@{|-}[l] \widetilde{k_*}}
\]
fits the situation of Day's Reflection Theorem \cite{Da2} and \cite[pages 24 and 25]{Da3}. For this, recall that a fully faithful functor $\Phi: \A \ra \X$ into a closed category $\X$ is said to be \emph{closed under exponentiation} when, for all $A$ in $\A$ and $X$ in $\X$, the internal hom $[X,\Phi A]$ is isomorphic to an
object of the form $\Phi B$ for some $B$ in $\A$.

\begin{theorem}
The functor $\textit{colim}(-,k_*): \Mky(G)_{\text{fin}} \ra \r$ is strong monoidal. Consequently,
$\widetilde{k_*}: \r \ra \Mky(G)_{\text{fin}}$ is monoidal and closed under exponentiation.
\end{theorem}

\begin{proof}
The first sentence follows quite formally from Lemma \ref{10.3} and the theory of Day convolution;
the main calculation is:
\begin{align*}
\text{\emph{colim}}(M*N,k_*)(Z) & = \int^C (M*N)(C) \ox_k k_* C \\
    & = \int^{C,X,Y} \T(X \times Y,C) \ox M(X) \ox_k N(Y) \ox_k k_*C \\
    & \cong \int^{X,Y} M(X) \ox_k N(Y) \ox_k k_*(X \times Y) \\
    & \cong \int^{X,Y} M(X) \ox_k N(Y) \ox_k k_*X \ox k_*Y \\
    & \cong \text{\emph{colim}}(M,k_*) \ox \text{\emph{colim}}(N,k_*).
\end{align*}
The second sentence then follows from \cite[Reflection Theorem]{Da2}.
\end{proof}

In fancier words, the adjunction
\[
\xymatrix@C=13pt{\text{\emph{colim}}(-,k_*)  & \ar@{|-}[l] \widetilde{k_*}}
\]
lives in the $2$-category of monoidal categories, monoidal functors and monoidal natural transformations
(all enriched over $\V$).


\section{Mackey functors for Hopf algebras} \label{Se11}

In this section we provide another example of a compact closed category $\T$ constructed from
a Hopf algebara $H$ (or \emph{quantum group}). We speculate that Mackey functors on this $\T$ will
prove as useful for Hopf algebras as usual Mackey functors have for groups.

Let $H$ be a braided (semisimple) Hopf algebra (over $k$). Let $\r$ denote the category of left $H$-modules
which are finite dimensional as vector spaces (over $k$). This is a compact closed braided monoidal
category.

We write $\Comod(\r)$ for the category obtained from the bicategory of that name in \cite{DMS} by taking
isomorphisms classes of morphisms. Explicitly, the objects are com\-onoids $C$ in $\r$. The morphisms
are isomorphism classes of comodules $S : \xymatrix@1@C=15pt{C \ar[r] |-{\object@{|}} & D}$ from $C$ to
$D$; such an $S$ is equipped with a coaction $\delta: S \ra C \ox S \ox D$ satisfying the coassociativity and
counity conditions; we can break the two-sided coaction $\delta$ into a left coaction $\delta_l:S \ra C \ox S$ and
a right coaction $\delta_r:S \ra S \ox D$ connected by the bicomodule condition. Composition of comodules
$S : \xymatrix@1@C=15pt{C \ar[r] |-{\object@{|}} & D}$ and
$T : \xymatrix@1@C=15pt{D \ar[r] |-{\object@{|}} & E}$ is defined by the (coreflexive) equalizer
\[
\xymatrix@C=7ex{
{S\ox_D T} \ar@<0.1ex>[r] & {S \ox T}
\ar@<0.7ex>[r]^-{1 \ox \delta_l} \ar@<-0.7ex>[r]_-{\delta_r \ox 1}
& {S \ox D \ox T~.} }
\]
The identity comodule of $C$ is $C : \xymatrix@1@C=15pt{C \ar[r] |-{\object@{|}} & C}$. The category
$\Comod(\r)$ is compact closed: the tensor product is just that for vector spaces equipped with the extra structure. Direct sums in $\Comod(\r)$ are given by direct sum as vector spaces. Consequently,
$\Comod(\r)$ is enriched in the monoidal category $\V$ of commutative monoids: to add comodules
$S_1 : \xymatrix@1@C=15pt{C \ar[r] |-{\object@{|}} & D}$ and
$S_2 : \xymatrix@1@C=15pt{C \ar[r] |-{\object@{|}} & D}$, we take the direct sum $S_1 \oplus S_2$ with coaction defined as the composite
\[
\xymatrix@C=7.5ex{
{S_1 \oplus S_2} \ar@<0.1ex>[r]^-{\delta_1\oplus \delta_2} &
{C \ox S_1 \ox D} \oplus {C \ox S_2 \ox D \cong C \ox(S_1 \oplus S_2) \ox D}.
 }
\]

We can now apply our earlier theory to the example $\T=\Comod(\r)$. In particular, we call a $\V$-enriched functor $M: \Comod(\r) \ra \Vect_k$ a \emph{Mackey functor on} $H$.

In the case where $H$ is the group algebra $kG$ (made Hopf by means of the diagonal
$kG \ra k(G \times G) \cong kG \ox_k kG)$, a Mackey functor on $H$ is not the same as a Mackey functor on
$G$. However, there is a strong relationship that we shall now explain.

As usual, let $\E$ denote the cartesian monoidal category of finite $G$-sets. The functor $k: \E \ra \r$ is strong monoidal and preserves coreflexive equalizers. There is a monoidal equivalence
\[
\Comod(\E) \simeq \Spn(\E),
\]
so $k: \E \ra \r$ induces a strong monoidal $\V$-functor
\[
\hat{k}: \Spn(\E) \ra \Comod(\r).
\]
With $\Mky(G) =[\Spn(\E), \Vect]_+$ as usual and with $\Mky(kG)= [\Comod(\r), \Vect]_+$,
we obtain a functor
\[
[\hat{k},1]: \Mky(kG) \ra \Mky(G)
\]
defined by pre-composition with $\hat{k}$. Proposition 1 of \cite{DS2} applies to yield:

\begin{theorem}
The functor $[\hat{k},1]$ has a strong monoidal left adjoint
\[
\exists_{\hat{k}}: \Mky(G) \ra \Mky(kG).
\]
The adjunction is monoidal.
\end{theorem}

The formula for $\exists_{\hat{k}}$ is
\[
\exists_{\hat{k}}(M)(R)= \int^{X \in \Spn(\E)} \Comod(\r)(\hat{k}X,R) \ox M(X).
\]

On the other hand, we already have the compact closed category $\r$ of finite-dimensional representations
of $G$ and the strong monoidal functor
\[
k_*: \Spn(\E)^{\text{op}} \ra \r.
\]
Perhaps $\r^{\text{op}} (\simeq \r)$ should be our candidate for $\T$ rather than the more complicated
$\Comod(\r)$. The result of \cite{DS2} applies also to $k_*$ to yield a monoidal adjunction
\[
\xymatrix{
[\r^{\text{op}}, \Vect]  \ar@<-1.2ex>[r]^-{\perp}_-{[k_*,1]} & \ar@<-1.2ex>[l]_-{\exists_{k^*}} \Mky(G).}
\]
Perhaps then, additive functors $\r^{\text{op}} \ra \Vect$ would provide a suitable generalization of
Mackey functors in the case of a Hopf algebra $H$. These matters require investigation at a later time.


\section{Review of some enriched category theory} \label{Se12}
The basic references are \cite{Ke}, \cite{La} and \cite{St}.

Let $\CO_\V$ denote the $2$-category whose objects are cocomplete $\V$-categories
and whose morphisms are (weighted-) colimit-preserving $\V$-functors; the $2$-cells
are $\V$-natural transformations.

Every small $\V$-category $\C$ determines an object $[\C,\V]$ of $\CO_\V$. Let
\[
Y: \C^{\text{op}} \ra [\C,\V]
\]
denote the Yoneda embedding: $YU = \C(U,-)$.

For any object $\X$ of $\CO_\V$, we have an equivalence of categories
\[
\CO_\V([\C,\V],\X) \simeq [\C^{\text{op}},\X]
\]
defined by restriction along $Y$. This is expressing the fact that $[\C,\V]$
is the free cocompletion of $\C^{\text{op}}$. It follows that, for small $\V$-categories
$\C$ and $\D$, we have
\begin{equation*}
\begin{split}
\CO_\V([\C,\V],[\D,\V]) & \simeq [\C^{\text{op}},[\D,\V]] \\
  & \simeq [\C^{\text{op}}\ox\D,\V].
\end{split}
\end{equation*}
The way this works is as follows.
Suppose $F: \C^{\text{op}} \ox \D \ra \V$ is a ($\V$-) functor. We obtain a
colimit-preserving functor
\[
\widehat{F}: [\C,\V] \ra [\D,\V]
\]
by the formula
\[
\widehat{F}(M)V = \int^{U\in \C} F(U,V) \ox MU
\]
where $M \in [\C,\V]$ and $V\in \D$. Conversely, given $G: [\C,\V] \ra [\D,\V]$,
define
\[
\Ve{G} : \C^{\text{op}} \ox \D \ra \V
\]
by
\[
\Ve{G}(U,V) = G(\C(U,-))V.
\]
The main calculations proving the equivalence are as follows:
\begin{equation*}
\begin{split}
\Ve{\widehat{F}}(U,V) & = \widehat{F}(\C(U,-))V \\
    & \cong \int^{U'}F(U',V) \ox \C(U,U') \\
    & \cong F(U,V) \quad \quad \text{by Yoneda; }
\end{split}
\end{equation*}
and,
\begin{equation*}
\begin{split}
\widehat{\Ve{G}}(M)V & = \int^{U} \Ve{G}(U,V)\ox MU \\
    & \cong (\int^{U}G(\C(U,-)) \ox MU)V \\
    & \cong G(\int^{U}\C(U,-)\ox MU)V \quad \text{since $G$ preserves weighted colimits} \\
    & \cong G(M)V \quad \quad \text{by Yoneda again.}
\end{split}
\end{equation*}

Next we look how composition of $G$s is transported to the $F$s. Take
\[
F_1: \C^{\text{op}} \ox \D \ra \V, \quad F_2: \D^{\text{op}} \ox \E \ra \V
\]
so that $\widehat{F_1}$ and $\widehat{F_2}$ are composable:

\[
\xygraph{
{[\C,\V]}="a"
(:[u(1.2)r(2.2)] {[\D,\V]}^-{\widehat{F_1}}
:[d(1.2)r(2.2)] {[\E,\V].}="t"^-{\widehat{F_2}},
"a" : @/_4ex/ _-{\widehat{F_2} \o \widehat{F_1}} "t"
)}
\]
Notice that
\begin{equation*}
\begin{split}
(\widehat{F_2} \o \widehat{F_1})(M)  & =\widehat{F_2}(\widehat{F_1}(M)) \\
    & = \int^{V\in \D} F_2(V,-) \ox \widehat{F_1}(M)V \\
    & \cong \int^{U,V} F_2(V,-) \ox F_1(U,V) \ox MU \\
    & \cong \int^{U} (\int^{V} F_2(V,-) \ox F_1(U,V)) \ox MU.
\end{split}
\end{equation*}
So we define $F_2 \o F_1 : \C^{\text{op}} \ox \E \ra \V$ by
\begin{equation} \label{1}
(F_2 \o F_1)(U,W) = \int^{V} F_2(V,W) \ox F_1(U,V);
\end{equation}
the last calculation then yields
\[
\widehat{F_2} \o \widehat{F_1} \cong \widehat{F_2 \o F_1}.
\]
The identity functor $1_{[\C,\V]} : [\C,\V] \ra [\C,\V]$ corresponds to the hom functor of $\C$;
that is,
\[
\Ve{1}_{[\C,\V]}(U,V) = \C(U,V).
\]

This gives us the bicategory $\VMod$. The objects are (small) $\V$-categories $\C$. A morphism
$F: \xymatrix@1{\C \ar[r] |-{\object@{/}} & \D}$
is a $\V$-functor $F: \C^{\text{op}} \ox \D \ra \V$; we call this a \emph{module from} $\C$ \emph{to} $\D$ (others call it a \emph{left} $\D$-,
\emph{right} $\C$-\emph{bimodule}). Composition of modules is defined by (\ref{1})
above.

We can sum up now by saying that
\[
\widehat{(~)} : \VMod \ra \CO_{\V}
\]
is a pseudofunctor (= homomorphism of bicategories) taking $\C$ to $[\C,\V]$, taking
$F: \xymatrix@1{\C \ar[r] |-{\object@{/}} & \D}$ to $\widehat{F}$,
and defined on $2$-cells in the obious way; moreover, this pseudofunctor is a local equivalence
(that is, it is an equivalence on hom-categories):
\[
\widehat{(~)}: \VMod (\C,\D) \simeq \CO_{\V}([\C,\V],[\D,\V]).
\]

A monad $T$ on an object $\C$ of $\VMod$ is called a \emph{promonad on} $\C$. It is the same as giving a
colimit-preserving monad
$\widehat{T}$ on the $\V$-category $[\C,\V]$. One way that promonads arise is from monoids $A$ for some
convolution monoidal structure on $[\C,\V]$; then
\[
\widehat{T}(M) = A \ast M.
\]
That is, $\C$ is a promonoidal $\V$-category \cite{Da}:
\[
P: \C^{\text{op}} \ox \C^{\text{op}}\ox \C \ra \V
\]
\[
J: \C \ra \V
\]
so that
\[
\widehat{T}(M) = A \ast M = \int^{U,V} P(U,V;-) \ox AU \ox MV.
\]
This means that the module $T: \xymatrix@1{\C \ar[r] |-{\object@{/}} & \C}$
is defined by
\begin{equation*}
\begin{split}
T(U,V)  & = \widehat{T}(\C(U,-))V \\
    & = \int^{U',V'} P(U',V';V) \ox AU' \ox \C(U,V') \\
    & \cong \int^{U'} P(U',U;V) \ox AU'.
\end{split}
\end{equation*}

A promonad $T$ on $\C$ has a unit $\eta: \Ve{1} \ra T$ with components
\[
\eta_{U,V} : \C(U,V) \ra T(U,V)
\]
and so is determined by
\[
\eta_{U,V}(1_U) : I \ra T(U,U),
\]
and has a multiplication $\mu: T \o T \ra T$ with components
\[
\mu_{U,W} : \int^{V} T(V,W) \ox T(U,V) \ra T(U,W)
\]
and so is determined by a natural family
\[
\mu'_{U,V,W}: T(V,W) \ox T(U,V) \ra T(U,W).
\]

The \emph{Kleisli category} $\C_T$ for the promonad $T$ on $\C$ has the same objects as $\C$ and has
homs defined by
\[
\C_T(U,V) = T(U,V);
\]
the identites are the $\eta_{U,V}(1_U)$ and the composition is the $\mu'_{U,V,W}$.

\begin{proposition} \label{13.1}
$[\C_T,\V] \simeq [\C,\V]^{\widehat{T}}.$
That is, the functor category $[\C_T,\V]$ is equivalent to the category of Eilenberg-Moore algebras for
the monad $\widehat{T}$ on $[\C,\V]$.
\end{proposition}

\begin{proof}(sketch)
To give a $\widehat{T}$-algebra structure on $M\in [\C,\V]$ is to give a morphism
$\alpha: \widehat{T}(M) \ra M$ satisfying the two axioms for an action. This is to give a natural
family of morphisms
\[
T(U,V) \ox MU \ra MV;
\]
but that is to give
\[
T(U,V) \ra [MU,MV];
\]
but that is to give
\begin{equation}\label{2}
\C_T(U,V) \ra \V(MU,MV).
\end{equation}
Thus we can define a $\V$-functor
\[
\overline{M}:\C_T \ra \V
\]
which agrees with $M$ on objects and is defined by (\ref{2}) on homs;
the action axioms are just what is
needed for $\overline{M}$ to be a functor. This process can be reversed.
\end{proof}

\section{Modules over a Green functor} \label{Se13}
In this section, we present work inspired by Chapters $2, 3$ and $4$ of \cite{Bo1}, casting it in a
more categorical framework.

Let $\E$ denote a lextensive category and $\CMon$ denote the category of commutative monoids;
this latter is what we called $\V$ in earlier sections. The functor $U:\kMod \ra $ $\CMon$ (which forgets the
action of $k$ on the $k$-module
and retains only the additive monoid structure) has a left adjoint $K: \CMon \ra \kMod$ which is strong monoidal for
the obvious tensor products on $\CMon$ and $\kMod$. So each category $\A$ enriched in $\CMon$ determines a category
$K_*\A$ enriched in $\kMod$: the objects of $K_*\A$ are those of $\A$ and the homs
are defined by
\[
(K_*\A)(A,B) = K\A(A,B)
\]
since $\A(A,B)$ is a commutative monoid. The point is that a $\kMod$-functor $K_*\A$
$ \ra \B$ is the same as a $\CMon$-functor $\A \ra U_*\B$.

We know that $\Spn(\E)$ is a $\CMon$-category; so we obtain a monoidal $\kMod$-category
\[
\C = K_*\Spn(\E).
\]

The $\kMod$-category of \emph{Mackey functors} on $\E$ is $\mathbf{Mky}_k(\E) = [\C,\kMod]$; it becomes monoidal using convolution with the monoidal structure on $\C$ (see Section \ref{Se5}). The $\kMod$-category of
\emph{Green functors} on $\E$ is $\mathbf{Grn}_k(\E) =\mathbf{Mon}[\C,\kMod]$ consisting of the monoids in
$[\C,\kMod]$ for the convolution.

Let $A$ be a Green functor. A \emph{module} $M$ over the Green functor $A$, or
\emph{$A$-module} means $A$ acts on $M$ via the convolution $\ast $. The monoidal
action $\alpha^M: A*M \ra M$ is defined by a family of morphisms
\[
\bar{\alpha}^M_{U,V} : A(U) \ox_k M(V) \ra M(U \times V) ,
\]
where we put $\bar{\alpha}^M_{U,V}(a \ox m)=a.m$ for $a \in A(U)$, $m \in M(V)$, satisfing
the following commutative diagrams for morphisms $f: U \ra U'$ and $g: V \ra V'$ in $\E$.
\[
\vcenter{\xygraph{
{A(U) \ox_k M(V)}="a"
(:[r(2.5)] {M(U \times V)}^-{\bar{\alpha}^M_{U,V}}
:[d(1.2)] {M(U' \times V')}="t"^-{M_*(f \times g)},
:[d(1.2)] {A(U') \ox_k M(V')}_-{A_*(f) \ox_k M_*(g)}
:"t"_-{\bar{\alpha}^M_{U',V'}}
)}}
\qquad
\vcenter{\xygraph{
{M(U)}="s"
(:[r(2.3)] {A(1) \ox_k M(U)}^-{\eta \ox 1}
:[d(1.2)] {M(1 \times U)}="t"^-{\bar{\alpha}^M},
:"t" _-{\cong}
)}}
\]

\[\xygraph{
{A(U) \ox_k A(V) \ox_k M(W)}="a"
(:[r(4.0)] {A(U) \ox_k M(V \times W)}^-{1 \ox \bar{\alpha}^M}
:[d(1.2)] {M(U \times V \times W)~.} ="t"^-{\bar{\alpha}^M},
:[d(1.2)] {A(U \times V) \ox_k M(W)}_-{\mu \ox 1}
:"t"_-{\bar{\alpha}^M}
)}
\]
If $M$ is an $A$-module, then $M$ is in particular a Mackey functor.

\begin{lemma} \label{10.1}
Let $A$ be a Green functor and $M$ be an $A$-module. Then $M_U$ is an $A$-module
for each $U$ of $\E$, where $M_U(X)=M(X \times U)$.
\end{lemma}

\begin{proof}
Simply define $\bar{\alpha}^{M_U}_{V,W }=\bar{\alpha}^M_{V,W \times U}$.
\end{proof}

Let $\ModA$ denote the category of left $A$-modules for a Green functor $A$. The objects are
$A$-modules and morphisms are $A$-module morphisms
$\theta: M \ra N$ (that is, morphisms of Mackey functors) satisfying the following commutative
diagram.
\[\xygraph{
{A(U) \ox_k M(V)}="a"
(:[r(2.5)] {M(U \times V)}^-{\bar{\alpha}^M_{U,V}}
:[d(1.2)] {N(U \times V)}="t"^-{\theta(U \times V)},
:[d(1.2)] {A(U) \ox_k N(V)}_-{1 \ox_k \theta(U)}
:"t"_-{\bar{\alpha}^N_{U,V}}
)}
\]
The category $\ModA$ is enriched in $\Mky$. The homs are given by the equalizer
\[\xygraph{
{\Mod(A)(M,N)}="a"
(:[r(2.5)] {\Hom(M,N)}
(:[r(3.4)] {\Hom(A*M,N)}="t"^-{\Hom(\alpha^M,1)},
:[d(1.2)r(1.8)] {\Hom(A*M,A*N)~.}_-{(A*-)}
:"t"_-{\Hom(1, \alpha^N)}
))}
\]
Then we see that $\ModA(M,N)$ is the sub-Mackey functor of $\Hom(M,N)$ defined by
\begin{equation*}
\begin{split}
\ModA(M,N)(U)= & \{ \theta \in \Mky(M(- \times U),N-) \quad  \mid~
\theta_{V \times W}(a.m)= a.\theta_W(m)  \\
& \quad \text{for all} ~ V,W,~ \text{and} ~ a \in A(V), m \in M(W \times U) \}.
\end{split}
\end{equation*}
In particular, if $A=J$ (Burnside functor) then $\ModA$ is the category of Mackey functors
and $\ModA(M,N)=\Hom (M,N)$.

The Green functor $A$ is itself an $A$-module. Then by the Lemma \ref{10.1}, we see that $A_U$
is an $A$-module for each $U$ in $\E$. Define a category $\C_A$ consisting of the objects of
the form $A_U$ for each $U$ in $\C$. This is a full subcategory of $\ModA$ and we have the
following equivalences
\[
\C_A(U,V) \simeq \ModA(A_U,A_V) \simeq A(U \times V).
\]

In other words, the category $\ModA$ of left $A$-modules is the category of
Eilenberg-Moore algebras
for the monad $T=A \ast -$ on $[\C,\kMod]$; it preserves colimits since it has a right adjoint
(as usual with convolution tensor products). By the above, the $\kMod$-category $\C_A$
(technically it is the Kleisli category
$\C_{\Ve{T}}$ for the promonad $\Ve{T}$ on $\C$; see Proposition \ref{13.1}) satisfies an equivalence
\[
[\C_A,\kMod] \simeq \ModA.
\]

Let $\C$ be a $\kMod$-category with finite direct sums and $\Omega$ be  a finite set of objects of $\C$ such that
every object of $\C$ is a direct sum of objects from $\Omega$.

Let $W$ be the algebra of $\Omega \times \Omega$ - matrices whose $(X,Y)$ - entry is a morphism $X \ra Y$ in $\C$. Then
\[
W=\left \{ (f_{XY})_{X,Y \in \Omega} ~\mid ~ f_{XY} \in \C(X,Y) \right \}
\]
is a vector space over $k$, and the product is defined by
\[
(g_{XY})_{X,Y \in \Omega} (f_{XY})_{X,Y \in \Omega} = \Big(\sum_{Y \in
\Omega} g_{YZ} \o f_{XY}\Big)_{X,Z \in \Omega}.
\]

\begin{proposition}
$[\C,\kMod] \simeq \kMod^{W}$ (= the category of left $W$-modules).
\end{proposition}

\begin{proof}
Put
\[
P=\bigoplus_{X \in \Omega} \C(X,-).
\]
This is a small projective generator so Exercise F (page 106) of \cite{Fr} applies and $W$
is identified as End(P).
\end{proof}

In particular; this applies to the category $\C_A$ to obtain the \emph{Green algebra} $W_A$ of a Green functor
$A$: the point being that $A$ and $W_A$ have the same modules.

\section{Morita equivalence of Green functors} \label{Se14}
In this section, we look at the Morita theory of Green functors making use of adjoint two-sided modules rather
than \emph{Morita contexts} as in \cite{Bo1}.

As for any symmetric cocomplete closed monoidal category $\W,$ we have the monoidal bicategory $\Mod(\W)$
defined as follows, where we take $\W=\Mky$. Objects are monoids $A$ in $\W$ (that is,
$A: \E \ra \kMod$ are Green functors) and morphisms are modules
$M: \xymatrix@1@C=15pt{A \ar[r] |-{\object@{|}} & B}$ (that is, algebras for the monad $A*-*B$ on $\Mky$)
with a two-sided action
\[
\infer{\bar{\alpha}^M_{U,V,W} : A(U) \ox_k M(V) \ox_k B(W) \ra M(U \times V \times W)}
      {\alpha^M : A * M * B \ra M}.
\]
Composition of morphisms $M: \xymatrix@1@C=15pt{A \ar[r] |-{\object@{|}} & B}$ and
 $N: \xymatrix@1@C=15pt{B \ar[r] |-{\object@{|}} & C}$ is $M*_B N$ and it is defined via the coequalizer
 \[
\xymatrix@C=7.5ex{
{M*B*N}
\ar@<0.7ex>[r]^-{\alpha^M *1_N} \ar@<-0.7ex>[r]_-{1_M *\alpha^N}
& {M*N} \ar@<0.1ex>[r]
& {M*_B N} = N \o M}
\]
that is,
\[(M*_B N)(U) = \sum_{X,Y}  \Spn(\E)(X \times Y, U) \ox M(X) \ox_k N(Y) / \sim_B~. \]
The identity morphism is given by $A: \xymatrix@1@C=15pt{A \ar[r] |-{\object@{|}} & A.}$

The 2-cells are natural transformations $\theta :M \ra M'$ which respect the actions
\[\xygraph{
{A(U) \ox_k M(V) \ox_k B(W)}="a"
(:[r(4.5)] {M(U \times V \times W)}^-{\bar{\alpha}^M_{U,V,W}}
:[d(1.2)] {M'(U \times V \times W)~.}="t"^-{\theta_{U \times V \times W}},
:[d(1.2)] {A(U) \ox_k M'(V) \ox_k B(W)}_-{1 \ox_k \theta_V \ox_k 1}
:"t"_-{\bar{\alpha}^{M'}_{U,V,W}}
)}
\]
The tensor product on $\Mod(\W)$ is the convolution $\ast$. The tensor product of the modules
$M: \xymatrix@1@C=15pt{A \ar[r] |-{\object@{|}} & B}$ and $N: \xymatrix@1@C=15pt{C \ar[r] |-{\object@{|}} & D}$ is
$M*N: \xymatrix@1@C=15pt{A*C \ar[r] |-{\object@{|}} & B*D}$.

Define Green functors $A$ and $B$ to be \emph{Morita equivalent} when they are equivalent in $\Mod(\W)$.

\begin{proposition}
If $A$ and $B$ are equivalent in $\Mod(\W)$ then $\ModA \simeq \ModB$ as categories.
\end{proposition}

\begin{proof}
$\Mod(\W)(-,J): \Mod(\W)^{\text{op}} \ra \CAT$ is a pseudofunctor and so takes equivalences to equivalences.
\end{proof}

Now we will look at the Cauchy completion of a monoid $A$ in a monoidal category
$\W$ with the unit $J$. The $\W$-category $\PA$ has underlying category $\Mod(\W)(J,A)= \Mod(A^{\text{op}})$
where $A^{\text{op}}$ is the monoid $A$ with commuted multiplication. The objects are
modules $M: \xymatrix@1@C=15pt{J \ar[r] |-{\object@{|}} & A}$; that is, right $A$-modules. The homs of $\PA$ are defined
by $(\PA)(M,N)= \Mod(A^{\text{op}})(M,N)$ (see the equalizer of Section \ref{Se13}).

The Cauchy completion $\QA$ of $A$ is the full sub-$\W$-category of $\PA$ consisting of the modules
$M: \xymatrix@1@C=15pt{J \ar[r] |-{\object@{|}} & A}$ with right adjoints
$N: \xymatrix@1@C=15pt{A \ar[r] |-{\object@{|}} & J}$. We will examine what the objects of $\QA$ are in more
explicit terms.

For motivation and preparation we will look at the monoidal category $\W = [\C, \S]$ where $(\C, \ox , I)$ is a monoidal category and $\S$ is the cartesian monoidal category of sets. Then $[\C, \S]$ becomes a monoidal
category by convolution. The tensor product $\ast$ and the unit $J$ are defined by
\begin{equation*}
\begin{split}
(M * N)(U) & = \int^{X,Y} \C(X \ox Y, U) \times M(X) \times N(Y) \\
J(U)& = \C(I,U).
\end{split}
\end{equation*}
Write $\Mod[\C,\S]$ for the bicategory whose objects are monoids $A$ in $[\C,\S]$ and whose morphisms are
modules $M: \xymatrix@1@C=15pt{A \ar[r] |-{\object@{|}} & B}$. These modules
have two-sided action
\[
\infer{\bar{\alpha}^M_{X,Y,Z} : A(X)  \times M(Y) \times B(Z) \ra M(X \ox Y \ox Z)~.}
      {\alpha^M : A * M * B \ra M}
\]
Composition of morphisms $M: \xymatrix@1@C=15pt{A \ar[r] |-{\object@{|}} & B}$ and
$N: \xymatrix@1@C=15pt{B \ar[r] |-{\object@{|}} & C}$ is given by the coequalizer
 \[
\xymatrix@C=7.5ex{
{M*B*N}
\ar@<0.7ex>[r]^-{\alpha^M *1_N} \ar@<-0.7ex>[r]_-{1_M *\alpha^N}
& {M*N} \ar@<0.1ex>[r]
& {M*_B N} }
\]
that is,
\[
(M*_B N)(U) = \sum_{X,Z} \C(X \ox Z, U) \times M(X) \times N(Z) / \sim_B
\]
where
\begin{align*}
(u, m \o b, n) & \sim_B (u, m, b \o n) \\
(t \o (r \ox s), m, n) & \sim_B (t, (Mr)m, (Ns)n)
\end{align*}
for ~ $ u: X \ox Y \ox Z \ra U,~ m \in M(X),~ b \in B(Y),~ n \in N(Z),~
t: X' \ox Z' \ra U, ~r: X \ra X',$  $s: Z \ra Z'$.

For each $K \in \C$, we obtain a module $A(K \ox -): \xymatrix@1@C=15pt{J \ar[r] |-{\object@{|}} & A}$. The action
\[
A(K \ox U) \ox A(V) \ra A(K \ox U \ox V)
\]
is defined by the monoid structure on $A$.

\begin{proposition}
Every object of the Cauchy completion $\QA$ of the monoid $A$ in $[\C, \S]$ is a retract
of a module of the form $A(K \ox -)$ for some $K \in \C$.
\end{proposition}

\begin{proof}
Take a module $M: \xymatrix@1@C=15pt{J \ar[r] |-{\object@{|}} & A}$ in $\Mod[\C, \S]$. Suppose that
$M$ has a right adjoint $N: \xymatrix@1@C=15pt{A \ar[r] |-{\object@{|}} & J}$. Then we have the following actions:
$ A(V) \times A(W) \ra A(V \ox W)$, $M(V) \times A(W) \ra M(V \ox W),
A(V) \times N(W) \ra N(V \ox W)$
since $A$ is a monoid, $M$ is a right $A$-module, and $N$ is a left $A$-module respectively.

We have a unit $\eta: J \ra M*_A N$ and a counit $\epsilon: N*M \ra A$ for the adjunction.
The component $\eta_U: \C(I,U) \ra (M*_A N)(U)$ of the unit $\eta$ is determined by
\[
\eta' = \eta_U(1_I) \in \sum_{X,Z} \C(X \ox Z, I) \times M(X) \times N(Z) / \sim_A;
\]
so there exist
$u: H \ox K \ra I, \quad p \in M(H), \quad q \in N(K)$ such that
$\eta'=[u,p,q]_A$.
Then
\[
\eta_u(f: I \ra U) =[fu: H \ox K \ra U, p, q]_A.
\]
We also have $\bar{\epsilon}_{Y,Z}: NY \times MZ \ra A(Y \ox Z)$ coming from $\epsilon$.
The commutative diagram
\[ \xygraph{
{M(U)}="s"
(:[r(5.3)] {\displaystyle {\sum_{X,Y,Z} \C(X \ox Y \ox Z, U) \times
M(X) \times N(Y) \times M(Z) / \sim}}^-{\eta_U * 1}
:[d(1.8)] {M(U)}="t"^-{1* \epsilon_U},
:"t" _-{1}
)}
\]
yields the equations
\begin{equation}\label{3}
\begin{split}
m & = (1*\epsilon_U)(\eta_U *1)(m) \\
    & = (1*\epsilon_U)[u \ox 1_U, p,q,m]_A \\
    & = M(u \ox 1_U)(p ~ \bar{\epsilon}_{K,U} (q,m))
\end{split}
\end{equation}
for all $m \in M(U)$.

Define
\[
\xymatrix@C=10ex{M(U)  \ar@/_/ [r] _{i_U} &
A(K \ox U) \ar@<-0.2ex>@/_/ [l] _{r_U}}
\]
by $i_U(m)=\bar{\epsilon}_{K,U} (q,m),  ~ r_U(a)=M(u \ox 1_U)(p.a)$.
These are easily seen to be natural in $U$. Equation (\ref{3}) says that $r \o i = 1_M$. So $M$ is a retract of
$A(K \ox -)$.
\end{proof}

Now we will look at what are the objects of $\QA$ when $\W$= $\Mky$ which is a symmetric monoidal closed, complete and cocomplete category.

\begin{theorem}
The Cauchy completion $\QA$ of the monoid $A$ in $\Mky$ consists of all the retracts
of modules of the form
\[
\bigoplus^k_{i=1} A(Y_i \times -)
\]
for some $Y_i \in \Spn(\E)$, $i=1,\ldots, k$.
\end{theorem}

\begin{proof}
Take a module $M: \xymatrix@1@C=15pt{J \ar[r] |-{\object@{|}} & A}$ in $\Mod(\W)$ and suppose that
$M$ has a right adjoint $N: \xymatrix@1@C=15pt{A \ar[r] |-{\object@{|}} & J}$.
For the adjunction, we have a unit $\eta: J \ra M*_A N$ and a counit $\epsilon: N*M \ra A$.
We write $\eta_U: \Spn(\E)(1,U) \ra (M*_A N)(U)$ is the component of the unit $\eta$ and it is determined by
\[
\eta' = \eta_1(1_1) \in \sum_{i=1}^k \Spn(\E)(X \times Y, 1) \ox M(X) \ox N(Y) / \sim_A.
\]
Put
\[
\eta' = \eta_1(1_1) = \sum_{i=1}^k [(S_i: X_i \times Y_i \ra 1) \ox m_i \ox n_i ]_A
\]
where $m_i \in M(X_i)$ and $n_i \in N(Y_i)$.
Then
\[
\eta_U(T: 1 \ra U) = \sum_{i=1}^k [(S_i \times T) \ox m_i \ox n_i ]_A.
\]
We also have $\bar{\epsilon}_{Y,Z}:NY \ox MZ \ra A(Y \times Z)$ coming from $\epsilon$.
The commutative diagram
\[ \xygraph{
{M(U)}="s"
(:[r(5.3)] {\displaystyle {\sum_{i=1}^k \Spn(\E)(X_i \times Y_i \times U, U) \ox
M(X_i) \ox N(Y_i) \ox M(U) / \sim_A}}^-{\eta_U * 1}
:[d(1.8)] {M(U)}="t"^-{1* \epsilon_U},
:"t" _-{1}
)}
\]
yields
\begin{equation*}
m = \sum_{i=1}^k [M(P_i \times U) \ox m_i \ox \epsilon(n_i \ox m) ]
\end{equation*}
where $m \in M(U)$ and $P_i: X_i \times Y_i \ra U$.

Define a natural retraction
\[
\xymatrix@C=10ex{M(U)  \ar@<-0.2ex>@/_/ [r] _-{i_U} &
{\displaystyle \bigoplus_{i=1}^k } A(Y_i \times U) \ar@/_/ [l] _-{r_U}}
\]
by
\[
r_U(a_i) = M(P_{i_k} \times U)(m_i.a_i), \quad
i_U(m) = \sum_{i=1}^k  \bar{\epsilon}_{Y_i ,U} (n_i \ox m).
\]
So $M$ is a retract of $\displaystyle{\bigoplus_{i=1}^k} A(Y_i \times -)$.

It remains to check that each module $A(Y \times -)$ has a right adjoint since retracts and direct
sums of modules with right adjoints have right adjoints.

In $\C = \Spn(\E)$ each object $Y$ has a dual (in fact it is its own dual). This implies that the module
$\C(Y,-): \xymatrix@1@C=15pt{J \ar[r] |-{\object@{|}} & J}$ has a right dual (in fact it is $\C(Y,-)$ itself) since
the Yoneda embedding $\C^{\text{op}} \ra [\C, \kMod]$ is a strong monoidal functor. Moreover, the unit
$\eta: J \ra A$ induces a module $\eta_*=A : \xymatrix@1@C=15pt{J \ar[r] |-{\object@{|}} & A}$ with a right
adjoint $\eta^* : \xymatrix@1@C=15pt{A \ar[r] |-{\object@{|}} & J}$. Therefore, the composite
\[
\xymatrix@C=7ex{J \ar[r] |-{\object@{|}}^{\C(Y,-)} & J \ar[r] |-{\object@{|}}^{\eta_*} & A},
\]
which is $A(Y \times -)$, has a right adjoint.
\end{proof}

\begin{theorem}
Green functors $A$ and $B$ are Morita equivalent if and only if $\QA \simeq \mathcal{Q} B$ as $\W$-categories.
\end{theorem}

\begin{proof}
See \cite{Li2} and \cite{St}.
\end{proof}



\end{document}